\documentclass[12pt]{amsart}
\usepackage{amsmath,amssymb,amsfonts,amsthm,amsopn}
\usepackage{graphicx,tikz}
\usepackage[all]{xy}
\usepackage{amsmath}
\usepackage{multirow, longtable, makecell, caption, array,enumitem}
\usepackage{amssymb}
\usepackage{mathtools}
\usepackage{bookmark}
\usepackage{hyperref}

\calclayout

\newtheorem{theorem}{Theorem}[section]
\newtheorem{lemma}[theorem]{Lemma}
\newtheorem{corollary}[theorem]{Corollary}
\newtheorem{proposition}[theorem]{Proposition}

\newtheorem{remark}[theorem]{Remark}

\theoremstyle{definition}

\usepackage{cellspace} %
\def\bR{\mathbb{R}}
\def\bC{\mathbb{C}}
\def\bN{\mathbb{N}}

\def\cF{\mathcal{F}}
\def\cS{\mathcal{S}}
\def\cA{S}

\def\cC{\mathcal{C}}

\def\rd{\bR^d}

\def\rdd{\bR^{2d}}

\def\la{\langle}
\def\ra{\rangle}
\def\lc{\left(}
\def\rc{\right)}
\def\lV{\left\lVert}
\def\rV{\right\rVert}

\def\wt{\widetilde}
\def\*b{*_{\bullet}}
\def\w{\mathrm{w}}

\def\S0{S^0_{0,0}}

\def\Bd'{B_{\delta'}}

\def\cBd'{\bar{B}_{\delta'}}

\def\Sp{\mathrm{Sp}(d,\bR)}
\def\Mp{\mathrm{Mp}(d,\bR)}
\def\SR{\mathrm{U}(2d,\bR)}
\def\OR{\mathrm{O}(2d,\bR)}
\def\diag{\mathrm{diag}}
\def\muS{\mu(S)}
\def\minftys{M^{\infty}_{v_s}(\rd)}
\newcommand{\GLL}{\mathrm{GL}(2d,\bR)}
\def\smo{\setminus\{0\}}

\begin{document}
	
\title[Dispersion, spreading and sparsity estimates for metaplectic operators]{Dispersion, spreading and sparsity of Gabor wave packets for  metaplectic and Schr\"odinger operators}
\author{Elena Cordero, Fabio Nicola and S. Ivan Trapasso}
\address{Dipartimento di Matematica ``G. Peano'', Università di Torino, via Carlo Alberto 10, 10123 Torino, Italy}
\address{Dipartimento di Scienze Matematiche ``G. L. Lagrange'', Politecnico di Torino, corso Duca degli Abruzzi 24, 10129 Torino, Italy}
\email{elena.cordero@unito.it}
\email{fabio.nicola@polito.it}
\email{salvatore.trapasso@polito.it}
%
\begin{abstract}
Sparsity properties for phase-space representations of several types of operators have been extensively studied in recent papers, including pseudodifferential, Fourier integral and metaplectic operators, with applications to time-frequency analysis of Schr\"odinger-type evolution equations. It has been proved that such operators are approximately diagonalized by Gabor wave packets. While the latter are expected to undergo some spreading phenomenon, there is no record of this issue in the aforementioned results. In this paper we prove refined estimates for the Gabor matrix of metaplectic operators, also of generalized type, where sparsity, spreading and dispersive properties are all noticeable. We provide applications to the propagation of singularities for the Schr\"odinger equation.
\end{abstract}
\maketitle
\section{Introduction and discussion of the results}
The relevance of the notion of wave packet in harmonic analysis and mathematical physics can be hardly overestimated. Roughly speaking, we say that a function $g$ on $\rd$ is a wave packet if it does possess good localization in the time-frequency space. To be more concrete, recall that good energy concentration of a function $g \in \cS(\rd)\smo$ (the Schwartz class) on a measurable set $T \subset \rd$ is achieved if there exists $0\le \delta_T \le 1/2$ such that
\[ \lc \int_{\rd \setminus T} |g(y)|^2 dy \rc^{1/2} \le \delta_T \lV g\rV_{L^2}. \] The spectral content of $g$ on a set $\Omega \subset \rd$ is well concentrated if the analogous estimate is satisfied by its Fourier transform $\widehat{g}$ for small $\delta_{\Omega}$. Therefore $g$ is concentrated on the cell $T\times \Omega$ in the phase space and the Donoho-Stark uncertainty principle prescribes a lower bound for the measure of such cell in terms of $\delta_T$ and $\delta_{\Omega}$ \cite{ds}. 

The essential time-frequency support of $g$ can be moved to $(x+T)\times (\xi+\Omega)$ for any choice of $(x,\xi) \in \rdd$ by applying a \textit{time-frequency shift} $\pi(x,\xi)=M_{\xi}T_x$ to $g$, namely as a result of the joint action of the modulation operator $M_{\xi}$ and the translation operator $T_x$, respectively defined as
\[ M_{\xi}g(y)= e^{2\pi iy \cdot \xi}g(y),\qquad T_{x}g(y)= g(y-x), \quad y \in \rd. \]
Functions of the type $\pi(z)g$ for some fixed $z \in \rdd$ and $g \in \cS(\rd)$ are called \textit{Gabor wave packets} or \textit{atoms}. In the case where $g(y)= e^{-\pi |y|^2}$ we speak of Gaussian wave packets; the latter are well-known textbook examples in physics. \\ 

Analysis of functions and operators in terms of Gabor wave packets is one of the primary purposes of modern time-frequency analysis \cite{cr book,gro book}. For instance, a phase-space representation of a signal $f \in L^2(\rd)$ is provided by the \textit{short-time Fourier transform}, which ultimately amounts to a decomposition of $f$ along the uniform boxes in phase space occupied by the Gabor atoms $\pi(z)g$, $z\in \rdd$, for some fixed window function $g \in \cS(\rd)\smo$. It is defined as
\[  V_{g}f (x,\xi)\coloneqq \langle f, \pi(x,\xi) g \rangle=\int_{\rd}e^{-2\pi iy \cdot \xi }
f(y)\, {\overline {g(y-x)}}\,dy, \quad (x,\xi)\in \rdd. \]
Time-frequency analysis of operators can be conducted along the same lines by investigating how they act at the atomic level. Precisely, the (continuous) \textit{Gabor matrix} of a linear continuous operator $A:\cS(\rd)\to \cS'(\rd)$ with respect to analysis and synthesis windows $g,\gamma \in \cS(\rd)\smo$ is defined by 
\[ K_A(w,z) \coloneqq \la A\pi(z)g, \pi(w)\gamma \ra, \quad w,z\in \rdd. \] It can be regarded as an infinite matrix encoding the phase-space features of $A$, since its action on the time-frequency space reads as an integral operator with kernel $K_A$: under the additional assumption $\lV g \rV_{L^2} = \| \gamma \|_{L^2}=1$ we have in fact the identity
\[ V_{\gamma}(Af)(w) = \int_{\rdd} K_A(w,z)V_{g} f(z)dz, \quad w \in \rdd. \]

It is clear that sparsity of $K_A$ is a highly desirable property, for both theoretical and numerical purposes. Several results concerning the approximate diagonalization of operators at the level of Gabor matrix have been appearing in the literature, in particular for pseudodifferential operators \cite{cnt alm diag,gro rze,gro sjo,rocht}, Fourier integral operators \cite{cgnr gen met,cgnr fio alg,cnr tfa int op} and propagators associated with Cauchy problems for Schr\"odinger-type evolution equations \cite{cnr gabor rep,cnr spars gabor}. We stress that wave packets should be tailored in order to best fit the geometry of the problem. For instance, the Gabor matrix of Fourier integral operators arising as propagators for strictly hyperbolic equations does not display a sparse behaviour, while analogous representations involving  curvelet atoms do enjoy super-polynomial decay, cf.\ \cite{candes,cordoba}. See also \cite{guo labate,smith,tataru} for other applications of wave packet analysis.

For the sake of concreteness let us focus on the Schr\"odinger propagator for the free particle $U(t) = e^{i(t/2\pi)\triangle}$, $t \in \bR$, and fix $g \in \cS(\rd)\smo$. For any $t\in \bR$ and $N \in \bN$ there exists a constant $C=C(t,N)>0$ such that the following decay estimate for the Gabor matrix elements of $U(t)$ holds:
\begin{equation} \label{intro diag est schr} |\la e^{i(t/2\pi)\triangle}\pi(z)g,\pi(w)g \ra| \le C (1+|w-\cA_t z|)^{-N}, \quad w,z \in \rdd, \end{equation}
where $\cA_t \in \bR^{2d\times 2d}$ is the block matrix
\begin{equation} \label{intro At free} \cA_t = \begin{bmatrix} I & 2tI \\ O & I \end{bmatrix},\end{equation} $I\in \bR^{d\times d}$ is the identity matrix and $O\in \bR^{d \times d}$ is the null matrix. We remark that $t\mapsto \cA_t$ coincides with the Hamiltonian flow for the free particle in phase space; precisely, the classical equations of motion with Hamiltonian $H(x,\xi)= |\xi|^2$ and initial datum $(x_0,\xi_0) \in \rdd$ are solved by $(x(t),\xi(t))=\cA_t(x_0,\xi_0)$.
Hence \eqref{intro diag est schr} shows that the time evolution of wave packets under $U(t)$ approximately follows the classical flow, in according with the correspondence principle of quantum mechanics. 

Nevertheless, a distinctive feature of wave propagation  dynamics is the unavoidable effect of diffraction. In the situation under our attention it does manifest itself as the well-known phenomenon of the spreading of wave packets. Moreover, a straightforward consequence of the dispersive estimates for the Schr\"odinger propagator \cite{tao book} is that there exists $C>0$ such that
\begin{equation} \label{intro disp est schr} |\la e^{i(t/2\pi)\triangle}\pi(z)g,\pi(w)g \ra| \le C (1+|t|)^{-d/2}, \quad w,z \in \rdd. \end{equation} It may therefore appear quite unsatisfactory that there is no trace of such issues in quasi-diagonalization estimates as \eqref{intro diag est schr}. The purpose of this paper is exactly to prove refined estimates for the Gabor matrix of $U(t)$ where sparsity, spreading and dispersive phenomena are fully represented. To the best of our knowledge, we are not aware of results in this spirit for pseudodifferential or evolution operators. 

Our quest is in fact motivated by the more general situation where $U(t)$ is the Schr\"odinger propagator corresponding to the Hamiltonian $H = Q^\w$, where $Q$ is a real homogeneous quadratic polynomial on $\rdd$ and $Q^\w$ denotes its Weyl quantization, (formally) defined as
\[ Q^\w f(x) = \int_{\rdd}e^{2\pi  i(x-y)\cdot \xi}Q\left(\frac{x+y}{2},\xi\right)f(y)dyd\xi. \]
For example, $(2\pi \xi_j)^\w = -i\partial_{x_j}$, $j=1,\ldots,d$, and $Q^\w = -\triangle$ for the choice $Q(x,\xi)= 4\pi^2|\xi|^2$. 

The propagator $U(t)=e^{-2\pi i t Q^\w}$, $t\in \bR$, is in turn an instance of a \textit{metaplectic operator}. In short, the metaplectic representation is a machinery which associates a symplectic matrix $S \in \Sp$ with a member of the metaplectic group $\mu(S)\in \Mp$, that is a unitary operator on $L^2(\rd)$ defined up to the sign. If $\bR \ni t \mapsto \cA_t\in \Sp$ denotes the classical flow on phase space associated with the quadratic Hamiltonian $H(x,\xi)=Q(x,\xi)$ then $\mu(\cA_t) = \pm e^{-2\pi i t Q^\w}$ - see \eqref{intro At free} for the free particle case. We refer to Section \ref{sec metap} below for further details and \cite{dg symp met,folland} for comprehensive discussions on the metaplectic representation.

It is therefore convenient to focus on metaplectic operators as primary objects of our investigation. The spreading of wave packets under $\mu(S)$ is now connected with the \textit{singular values} of $S\in \Sp$ \cite{cauli}, which occur in couples $(\sigma,\sigma^{-1})$ of positive real numbers. We fix the ordering by labelling the largest $d$ singular values in such a way that $\sigma_1 \ge \ldots \ge \sigma_d \ge 1$; moreover we set $\Sigma=\diag(\sigma_1,\ldots,\sigma_d)$ and introduce the matrices 
\[ \quad D=\begin{bmatrix} \Sigma & O \\ O & \Sigma^{-1} \end{bmatrix}, \quad D'=\begin{bmatrix} \Sigma^{-1} & O \\ O & I \end{bmatrix}, \quad D''=\begin{bmatrix} I & O \\ O & \Sigma^{-1} \end{bmatrix}. \] The singular value decomposition of $S \in \Sp$ (also known as \textit{Euler decomposition} in this setting) has a peculiar form due to the symplectic condition, namely there exist (non-unique) orthogonal and symplectic matrices $U,V$ such that $S=U^\top D V$, cf.\ Proposition \ref{euler dec} below. Such factorization is identified by the triple $(U,V,\Sigma)$. In the following for a given $S\in \Sp$ we will denote by $(U,V,\Sigma)$ an Euler decomposition of $S$ and by $D, D', D''$ the above defined related matrices.

We are now in the position to state our first result, concerning rapidly decaying Gabor wave packets. \par
\begin{theorem}\label{maint schw}
	 For any $g, \gamma \in \cS(\rd)$ and $N >0$ there exists $C>0$ such that, for every $S\in\Sp$,
	\begin{equation}\label{maint estimate}
	\left| \la \muS \pi(z)g, \pi(w)\gamma \ra \right| \le C (\det \Sigma)^{-1/2} (1+|D'U(w-Sz)|)^{-N}, \quad z,w \in \rdd. 
	\end{equation}
\end{theorem} 
We see that the simultaneous occurrence of sparsity, spreading and dispersive phenomena are represented by the quasi-diagonal structure along $S$, the dilation by $D'U$ and the factor $(\det\Sigma)^{-1/2}$ respectively. An equivalent form of the previous estimate where the spreading phenomenon is somehow more distributed follows by noticing that $D'U(w-Sz)=D'Uw-D''Vz$. 

The special case of the free particle propagator is treated in detail in Section \ref{sec free prop} below. We just mention here that, for any fixed $t\in \bR$ and any Euler decomposition $(U_t,V_t,\Sigma_t)$ of $\cA_t$, the estimate \eqref{maint estimate} reads
\[ 	\left| \la e^{i(t/2\pi)\triangle} \pi(z)g, \pi(w)\gamma \ra \right| \le C (1+|t|)^{-d/2} (1+|  D'_tU_t(w-\cA_t z)|)^{-N}, \quad w,z\in \rdd. \]
We see that the features of both \eqref{intro diag est schr} and \eqref{intro disp est schr} are now represented, whereas the spreading phenomenon manifests itself as a dilation by the matrix $D'_tU_t$, the nature of which is investigated in Section \ref{sec free prop}. 

We provide results in the same spirit of Theorem \ref{maint schw} for wave packets associated with less regular atoms; in particular we assume that $g$ and $\gamma$ satisfy certain phase-space decay conditions. The function spaces arising by imposing some weighted Lebesgue regularity on the short-time Fourier transform of a function are of primary concern in time-frequency analysis and are known as \textit{modulation spaces} \cite{F1}. To be precise, let $1\le p < \infty$ and $s\in \bR$, and define the polynomial weight function $v_s(z)\coloneqq (1+|z|)^s$ on $\rdd$. The modulation space $M^p_{v_s}(\rd)$ is defined as the subset of temperate distributions $f \in \cS'(\rd)$ such that, for any $g \in \cS(\rd)\smo$, 
\[ \lV f \rV_{M^p_{v_s}}\coloneqq \lc \int_{\rdd} |\la f,\pi(z)g \ra|^p v_s(z)^p dz\rc^{1/p} < \infty. \] Similarly, we say that $f \in \minftys$ if there exists $C>0$ such that, for any $g \in \cS(\rd)\smo$, \[ |\la f, \pi(z)g \ra | \le C (1+|z|)^{-s}, \quad z \in \rdd.\] We write $M^p(\rd)$ for the unweighted case ($s=0$). 

We collect some of the properties of modulation spaces in Proposition \ref{mod sp prop} below. We just recall that they provide a refined framework of (Banach) spaces which encompasses several classical spaces of real harmonic analysis. For instance, we have that $M^2_{v_s}(\rd)$, $s\in \bR$, coincides with the Shubin-Sobolev space of order $s$ \cite{shubin}, namely \[ Q^s(\rd) = L^2_{v_s}(\rd)\cap H^s(\rd) = \{ f \in \cS'(\rd) : f,\widehat{f} \in L^2_{v_s}(\rd)\}.\]  Note in particular that $M^2(\rd)=L^2(\rd)$. Moreover, the modulation spaces $M^p_{v_s}$ are related to the Schwartz class (and its dual space $\cS'$) via the following characterizations, for every $1\le p \le \infty$,
\begin{equation}\label{schw char}
\cS(\rd)= \bigcap_{s\ge 0} M^p_{v_s}(\rd), \quad \cS'(\rd)=\bigcup_{s\ge 0} M^p_{v_{-s}}(\rd). 
\end{equation}
Modulation spaces also provide an optimal environment where to investigate the behaviour of the Gabor matrix of a metaplectic operator, as evidenced by the following result.   
\begin{theorem}\label{maint mp}
\begin{enumerate}[label=(\roman*)]
\item Let $1\le p,q,r \le \infty$ satisfy $1/p + 1/q = 1+ 1/r$.	For any $g \in M^p(\rd)$, $\gamma \in M^q(\rd)$, $S\in\Sp$, there exists $H \in L^r(\rdd)$ such that, for any $z,w \in \rdd$, 
	\begin{equation}\label{eq maint mp}
	\left| \la \muS \pi(z)g, \pi(w)\gamma \ra \right| \le H(D'U(w-Sz)),
	\end{equation} 
	with 
	\begin{equation}\label{eq norm maint mp} \lV H \rV_{L^r} \le  (\det\Sigma)^{1/2-1/r} \lV g \rV_{M^p} \lV \gamma \rV_{M^q}. \end{equation}
	\item Let $s>2d$. For any $g,\gamma \in \minftys$  there exists $H \in L^{\infty}_{v_{s-2d}}(\rdd)$ such that \eqref{eq maint mp} holds, with 
	\begin{equation}\label{eq norm maint minftys} \lV H \rV_{L^{\infty}_{v_{s-2d}}} \le (\det\Sigma)^{-1/2} \lV g \rV_{\minftys} \lV \gamma \rV_{\minftys}. \end{equation}
\end{enumerate}
\end{theorem} Here we used the notation $\| H \|_{L^\infty_{v_s}} = \| H v_s \|_{L^\infty}$. We remark that the best decay in \eqref{eq norm maint mp} is achieved in the case where $p=q=r=1$, namely for Gabor atoms belonging to the modulation space $M^1(\rd)$ - the so-called \textit{Feichtinger algebra} \cite{F2 new segal}. We also highlight the inclusion $\minftys \subset M^1(\rd)$ for $s>2d$, which follows directly from the definition.

In Theorem \ref{maint gen met} we prove an estimate in the same spirit of Theorem \ref{maint mp} for the Gabor matrix of the so-called \textit{generalized metaplectic operators}. This family of operators characterized by the sparsity of their phase-space representation has been introduced and studied in \cite{cgnr gen met,cgnr fio alg} in connection with inverse-closed algebras of Fourier integral operators. Their main properties are recalled in Section \ref{sec metap}.

Finally, we provide an application of the enhanced estimates for the Gabor matrix to the propagation of singularities for the Schr\"odinger equation. The fruitful interplay between time-frequency and microlocal analysis lead to new notions of \textit{global} wave front sets after H\"ormander \cite{hormander quadratic}. Several notions of global wave front set have been introduced to detect (lack of) regularity at modulation space level, see \cite{rw} for a more detailed discussion and \cite{pravda,wahlberg} for further applications. 

Given an open cone $\Gamma$ in $\rdd$ and $g \in \cS(\rd)\smo$ we define the space of $M^1_{(g)}(\Gamma)$ of $M^1$-regular distributions on the cone $\Gamma$ with respect to $g$ as the set of all $f \in \cS'(\rd)$ such that
\begin{equation}\label{M1 cone def} \lV f \rV_{M^1_{(g)}(\Gamma)} \coloneqq \int_{\Gamma} |V_g f(z)| dz < \infty. \end{equation}

The next result shows that $M^1$-regularity of a function $f$ on a conic subset of the phase space is preserved by the action of $\mu(S)$ provided that the cone evolves under $S$. We set $\mathbb{S}^{2d-1}$ for the sphere in $\rdd$. 

\begin{theorem}\label{maint prop sing}
	Let $S \in \Sp$, $g,\gamma \in \cS(\rd)\smo$  and $\Gamma,\Gamma' \subset \rdd$ be open cones such that $\overline{\Gamma'\cap \mathbb{S}^{2d-1}} \subset \Gamma\cap \mathbb{S}^{2d-1}$. If $f\in \cS'(\rd)$ is  $M^1$-regular on $\Gamma$ with respect to $g$ then $\mu(S) f$ is $M^1$-regular on $S(\Gamma')$ with respect to $\gamma$. \par
Precisely, given $r\ge 0$ there exists $C>0$ such that, for any $f \in M^1_{v_{-r}}(\rd)\cap M^1_{(g)}(\Gamma)$ (cf.\ \eqref{schw char}) and $S\in\Sp$ the following estimate holds:
\[ \lV \muS f \rV_{M^1_{(\gamma)}(S (\Gamma'))} \le C (\det\Sigma)^{1/2} \lc \lV f \rV_{M^1_{(g)}(\Gamma)} + (\det\Sigma)^{r}\lV f \rV_{M^1_{v_{-r}}(\rd)} \rc. \]
\end{theorem}

If we specialize the previous result to the free particle propagator we get \[ \lV e^{i(t/2\pi)\triangle} f \rV_{M^1_{(\gamma)}(\cA_t (\Gamma'))} \le C \lc (1+|t|)^{d/2}  \lV f \rV_{M^1_{(g)}(\Gamma)} + (1+|t|)^{d(1/2+r)} \lV f \rV_{M^1_{v_{-r}}(\rd)} \rc, \]
where $\cA_t$ is the classical flow in \eqref{intro At free}. The latter can be regarded as a microlocal refinement of known estimates, cf.\ for instance \cite[Prop. 6.6]{wang} and Corollary \ref{cor34} below.\par\bigskip
In short, the paper is organised as follows. In Section 2 we collect some auxiliary results. Section 3 is devoted to the proof of the main results. Section 4, as already anticipated, provides the example of the Schr\"odinger free propagator in detail.

\section{Preliminaries}
\subsection{Notation} We set $|t|^2=t\cdot t$, $t\in\rd$, where $x\cdot y$ is the scalar product on $\rd$. The bracket  $\la  f,g\ra$ denotes the extension to $\cS' (\rd)\times \cS (\rd)$ of the inner product $\la f,g\ra =\int_{\rd} f(t){\overline {g(t)}}dt$ on $L^2(\rd)$. 

The conjugate exponent $p'$ of $p\in [1,\infty]$ is defined by $1/p+1/p'=1$ if $1\le p < \infty$ and as $p'=1$ if $p=\infty$. The symbol $\lesssim$ means that the underlying inequality holds up to a universal positive constant factor; the latter may possibly depend on some ``allowable'' parameter $\lambda$, in which case we write
\[ f\lesssim_{\lambda} g\quad\Rightarrow\quad\exists C=C(\lambda)>0\,:\,f\le Cg. \] Moreover, $f\asymp g$ stands for the case where both $f \lesssim g$ and $g\lesssim f$ hold. 

The characteristic function of a set $A\subset \rd$ is denoted by $1_A$. Recall that $\Gamma \subset \rd$ is a conic subset of $\rd$ if it is invariant under multiplication by positive real numbers, namely $x \in \Gamma \Rightarrow \lambda x \in \Gamma$ for any $\lambda >0$. 

We choose the following normalization for the Fourier transform:
\[ \cF(f)(\xi) = \hat{f}(\xi) = \int_{\rd} e^{-2\pi i x\cdot \xi} f(x)dx, \quad \xi \in \rd.
\] 

The reflection operator is defined as $f^{\vee}(t) = f(-t)$,  $t \in \rd$. 

Given $A,B \in \bR^{d\times d}$ the direct sum $A \oplus B \in \bR^{2d \times 2d}$ is defined as
\[ A\oplus B = \diag(A,B) = \begin{bmatrix} A & O \\ O & B \end{bmatrix}. \]

\subsection{Symplectic matrices}
The canonical symplectic matrix $J \in \bR^{2d \times 2d}$ is defined as
\[ J = \begin{bmatrix}
O & I \\ - I & O
\end{bmatrix}. \] 
The symplectic group $\Sp$ is defined by
\[ \Sp = \{ S \in \GLL \,:\, S^\top J S = J \}. \] 
Recall that the complex unitary group $U(d,\bC)$ is isomorphic to the subgroup of \textit{symplectic rotations} $\SR$ of $\Sp$ \cite{dg symp met}, namely 
\[ \SR = \Sp \cap \OR. \] An equivalent, more concrete representation of symplectic rotations is 
\[ \SR = \left\{ \begin{bmatrix}
A & -B \\ B & A
\end{bmatrix} \in \bR^{2d\times 2d} \, : \, AA^\top +BB^\top = I, \, AB^\top = B^\top A \right\}. \] 

We recall a result on a SVD-like decomposition of symplectic matrices, also known as the \textit{Euler decomposition} in the literature; see \cite[Appendix B.2]{serafini} for details and proofs. 

\begin{proposition} \label{euler dec}
For any $S \in \Sp$ there exist $U,V \in \SR$ such that \[ S=U^\top DV, \quad D = \Sigma \oplus \Sigma^{-1},\]
where $\Sigma=\diag(\sigma_1,\ldots,\sigma_d)$ and $\sigma_1 \ge \ldots \ge \sigma_d \ge \sigma_d^{-1} \ge \ldots \ge \sigma_1^{-1}$ are the singular values of $S$. 
\end{proposition}
We stress that while $\Sigma$ is uniquely determined for given $S$ once the order of the singular values is fixed, the matrices $U$ and $V$ appearing in such factorization are not unique in general due to possible occurrence of degenerate singular values. 

We identify any Euler decomposition of $S$ as $U^\top DV$ with the triple $(U,V,\Sigma)$. 

Recall from the Introduction that other useful related matrices are \begin{equation}\label{def D'D''} D' = \Sigma^{-1}\oplus I, \quad D'' = I\oplus\Sigma^{-1}. \end{equation}

\subsection{Modulation spaces}\label{sec mod sp}
We provide a collection of time-frequency analysis tools that are needed in the paper. The reader may consult \cite{gro book} for further details and proofs of the mentioned results. 

Recall that the short-time Fourier transform (STFT) of a temperate distribution $f\in\cS'(\rd)$ with respect to the window function $g \in \cS(\rd)\setminus\{0\}$ is defined as
\begin{equation}\label{FTdef}
V_gf (x,\xi)\coloneqq \langle f, \pi(x,\xi) g\rangle=\int_{\rd}e^{-2\pi iy \cdot \xi }
f(y)\, {\overline {g(y-x)}}\,dy.
\end{equation}

The STFT is intimately connected with other well-known phase-space transforms such as the \textit{Wigner distribution}
\begin{equation} W(f,g)(x,\xi )=\int_{\mathbb{R}^{d}}e^{-2\pi iy\cdot \xi }f\left(x+\frac{y}{2}\right)%
\overline{g\left(x-\frac{y}{2}\right)}\ dy. \label{wig def} \end{equation}
We write $Wf$ when $f=g$. In particular, we have
\begin{equation}\label{stft rel} W(f,g) = 2^d e^{4\pi i x\cdot \xi} V_{g^\vee}f(2x,2\xi). \end{equation}
We also recall the orthogonality identity (also known as \textit{Moyal formula} for the Wigner distribution):
\begin{equation} \label{moyal}
\la V_{g_1}f_1, V_{g_2}f_2 \ra_{L^2} = \la W(f_1,g_1), W(f_1,g_1) \ra_{L^2} = \la f_1,f_2 \ra_{L^2} \overline{\la g_1,g_2 \ra_{L^2}},
\end{equation} 
for any $f_1,g_1,f_2,g_2 \in L^2(\rd)$. 
The behaviour of the Wigner distribution under time-frequency shifts is given by
\begin{equation} \label{wig cov}
W(\pi(w)f,\pi(z)g)(u) = c(w,z) M_{J(w-z)}T_{\frac{w+z}{2}} W(f,g)(u), \quad u,w,z \in \rdd,
\end{equation}
where $c(w,z) = e^{\pi i (w_1+z_1)\cdot(z_2-w_2)}$. The identity $W(\pi(z)f)(u) = Wf(u-z)$ is often referred to as the \textit{covariance property} of $Wf$.  \\

Recall from the Introduction that, for $1\le p \le \infty$ and $s\in \bR$, the modulation space $M^p_{v_s}(\rd)$ is the set of $f \in \cS'(\rd)$ such that, for any $g \in \cS(\rd)\smo$, 
\begin{equation}\label{mod sp norm} \lV f \rV_{M^p_{v_s}} \coloneqq \lV V_g f \rV_{L^p_{v_s}} = \lc \int_{\rdd}|V_g f(z)|^p v_s(z)^p dz \rc^{1/p} < \infty,\end{equation} with trivial modification in the case $p=\infty$. We collect below the relevant properties of modulation spaces that will be repeatedly used in the rest of the paper, see \cite{dg symp met,gro book} for proofs and generalizations. 
\begin{proposition} \label{mod sp prop} Consider $1\le p \le \infty$ and $s,r\in \bR$ such that $|s| \le r$. \begin{enumerate}[label=(\roman*)]
		\item $M^p_{v_s}(\rd)$ is a Banach space with the norm \eqref{mod sp norm}, which is independent of the window function $g$ (in the sense that different windows yield equivalent norms). Moreover, the class of admissible non-zero windows can be extended from $\cS(\rd)$ to $M^1_{v_r}(\rd)$.
		\item If $p< \infty$ the Schwartz class $\cS(\rd)$ is dense in $M^p_{v_s}(\rd)$. Moreover, for any $g \in \cS(\rd)\smo$, 
		\[ f \in \cS(\rd) \Longrightarrow V_g f, \, Wf \in \cS(\rdd). \]
		\item If $p_1\le p_2$ and $s_2 \le s_1$, then $M^{p_1}_{v_{s_1}}(\rd) \subseteq M^{p_2}_{v_{s_2}}(\rd)$. In particular, for $|s|\le r$,
		\[ M^1_{v_r}(\rd) \subseteq M^p_{v_s}(\rd) \subseteq M^{\infty}_{v_{-r}}(\rd). \]
		\item If $1\le p < \infty$ then $\lc M^p_{v_s}(\rd)\rc'\simeq M^{p'}_{v_{-s}}(\rd)$ and the duality is concretely given by 
		\[ \la f,h \ra = \int_{\rdd} V_{g}f(z) \overline{V_{g}h(z)} dz, \] for $f \in M^p_{v_s}(\rd)$, $h \in M^{p'}_{v_{-s}}(\rd)$ and $g \in M^1_{v_r}(\rd)$ with $\lV g \rV_{L^2} = 1$. 
	\end{enumerate}
\end{proposition}

It turns out that the STFT is injective in $\cS'(\rd)$, as a result of the following \textit{inversion formula}: for $f \in \cS'(\rd)$ and $g,\gamma \in \cS(\rd)\smo$ such that $\la g,\gamma \ra \ne 0$ we have
\begin{equation}\label{rec form cont}
f = \frac{1}{\la \gamma,g \ra}  \int_{\rdd} V_g f (z) \pi(z)\gamma dz, \end{equation} in the sense of temperate distributions. The same result extends to the case where $f \in M^p_{v_s}(\rd)$ and $g,\gamma \in M^1_{v_r}(\rd)\smo$ with $s$ and $r$ as in Proposition \ref{mod sp prop}. The particular choice $\gamma=g$ yields
\begin{equation}\label{id mod sp} \mathrm{Id}_{M^p_{v_s}} = \frac{1}{\lV g \rV_{L^2}^{2}} V_g^* V_g, \end{equation}
where $V_g^*$ is the adjoint STFT defined as a vector-valued integral by 
\[ V_g^* F = \int_{\rdd} F(z)\pi(z)g dz, \quad F \in \cS(\rdd). \] 
Moreover, for $g \in M^1_{v_r}(\rd)\smo$ we have that $V_g:M^p_{v_s}(\rd)\to L^p_{v_s}(\rdd)$ and $V_g^*:L^p_{v_s}(\rdd) \to M^p_{v_s}(\rd)$ are continuous maps. 

The inversion formula enables an efficient phase-space analysis of operators as already mentioned in the Introduction. Consider a bounded linear operator $A:\cS(\rd)\to\cS'(\rd)$ and $g,\gamma \in \cS(\rd)\smo$; it is not restrictive to assume $\lV g \rV_{L^2} = \lV \gamma \rV_{L^2}=1$. Using \eqref{id mod sp} we have that 
\[ A = V_{\gamma}^* V_{\gamma} A V_{g}^* V_{g}= V_{\gamma}^*\wt{A}V_{g}, \]
where $\wt{A}\coloneqq V_{\gamma} A V_{g}^*$ is an integral operator in $\rdd$ with integral kernel given by the \textit{Gabor matrix} $K_A$, that is
\begin{equation}\label{gabor ker} \wt{A}F(w) = \int_{\rdd} K_A(w,z)F(z)dz, \quad  K_A(w,z) = \la A\pi(z)g,\pi(w)\gamma \ra, \quad w \in \rdd. \end{equation}

\subsection{Weyl operators}\label{sec weyl}
Given $a \in \cS'(\rdd)$ (\textit{symbol}) the corresponding Weyl operator $a^\w : \cS(\rd) \to \cS'(\rd)$ is defined by duality as
\[ \la a^\w f,g\ra = \la a, W(g,f)\ra, \quad f,g\in \cS(\rd), \]
where $W(g,f)$ is the Wigner distribution introduced in \eqref{wig def}. 
Modulation spaces have been extensively employed as symbol classes as well as background spaces where to study boundedness of Weyl operators. Among the several results in this respect we highlight the special properties of Weyl operators with symbols in $M^{\infty,1}(\rdd)$ - the so-called \textit{Sj\"ostrand class} after \cite{sjo}. It is a modulation space of more general form than above, since its norm involves a mixed Lebesgue regularity condition on the short-time Fourier transform of a distribution: for any $g\in \cS(\rd)\smo$, 
\[ \lV f \rV_{M^{\infty,1}}\coloneqq \lV V_g f \rV_{L^{\infty,1}} = \int_{\rd} \sup_{x\in \rd} |V_gf (x,\xi)| d\xi < \infty. \]

We list below some results first proved in \cite{gro sjo}, see also \cite{bcgt alm diag,cnt alm diag,ct 20} for generalizations and further results on almost diagonalization of operators.  
\begin{theorem}\label{weyl properties}
Fix $g,\gamma \in M^1(\rd)$ and consider $a \in \cS'(\rdd)$. We have that $a \in M^{\infty,1}(\rdd)$ if and only if there exists a function $H \in L^1(\rdd)$ such that
\[ |\la a^\w \pi(z)g,\pi(w)\gamma \ra| \le H(w-z), \quad z,w \in \rdd. \]
The controlling function $H$ can be chosen as 
\[ H(w) = \sup_{z \in \rdd} |V_{\Phi}a(z,w)|,\quad \Phi = W(\gamma,g),  \]
hence $\lV H \rV_{L^1} \asymp \lV a \rV_{M^{\infty,1}}$. 
Moreover, $a^\w$ is bounded on any modulation space $M^p(\rd)$, $1 \le p \le \infty$.
\end{theorem}

\subsection{Metaplectic operators}\label{sec metap}
Recall that the \textit{metaplectic group} $\Mp$ is the universal double cover of the symplectic group $\Sp$. The corresponding faithful, strongly continuous unitary representation in $L^2(\rd)$ allows us to directly interpret $\Mp$ as a subgroup of $\mathcal{U}(L^2(\rd))$, hence consisting of \textit{metaplectic operators}. We use the notation $\muS$ to denote metaplectic operators defined up to sign, where $S = \rho^{\mathrm{Mp}}(\muS) \in \Sp$ and $\rho^{\mathrm{Mp}}:\Mp\to\Sp$ is the group projection, hence 
\[ \mu(AB) = \pm \mu(A)\mu(B), \quad A,B\in \Sp. \] 
An operator $\mu(S)$ satisfies the intertwining relation
\[
\pi(S z)=\mu(S) \pi(z) \mu(S)^{-1},\quad z \in \rdd.
\]
We provide some elementary examples of metaplectic operators which are associated with special elements of $\Sp$. In fact, it turns out that the metaplectic group is in some sense generated by operators $\mu(J)$, $\mu(\cA)$ and $\mu(\cC)$ defined below, cf.\ \cite{dg symp met} for a precise account.
\begin{enumerate}
	\item The Fourier transform is a metaplectic operator associated with the canonical symplectic matrix, that is $\mu(J)f = \pm \cF(f)$. Notice in particular that $\mu(-J)= \pm \cF^{-1}$. 
	\item Consider $A \in \mathrm{GL}(d,\bR)$ and set 
	\[ \cA = \begin{bmatrix} A & O \\ O & (A^{-1})^\top \end{bmatrix}. \]
	The metaplectic operator $\mu(\cA)$ acts as a rescaling by $A$:
	\[ \mu(\cA)f(t) = \pm |\det A|^{-1/2}f(A^{-1}t). \]
	\item Let $C \in \bR^{d\times d}$ be a real symmetric matrix and set 
	\[ \cC = \begin{bmatrix} I & O \\ C & I \end{bmatrix}. \]
	The metaplectic operator $\mu(\cC)$ is a chirp multiplication:
	\[ \mu(\cC)f(t) = \pm e^{\pi i t\cdot Ct} f(t). \]
\end{enumerate} 

We already mentioned that an important example of metaplectic operator is provided by the Schr\"odinger propagator for the free particle $U(t)=e^{i(t/2\pi)\triangle}$, $t\in \bR$. This can be easily derived from the examples above since $U(t)$ is a Fourier multiplier with chirp symbol $m_t(\xi) = e^{-2\pi i t |\xi|^2}$ on $\rd$, hence 
\begin{equation}\label{schr prop} U(t) = \cF^{-1}m_t\cF = \pm \mu(\cA_t), \quad  \cA_t= \begin{bmatrix}
I & 2t I \\ O & I
\end{bmatrix}, \quad t \in \bR. \end{equation}

A distinctive property of the Weyl calculus is known as \textit{symplectic covariance} {\cite[Thm.\ 215]{dg symp met}}: for any $S \in \Sp$ and $a \in \cS'(\rdd)$,
\begin{equation}\label{symp cov}
(a \circ S)^\w = \mu(S)^{-1}a^\w \mu(S).
\end{equation} 

Metaplectic operators have been thoroughly studied in the framework of phase-space analysis \cite{dg symp met,folland} and also in connection with the Schr\"odinger equation with quadratic Hamiltonians \cite{cnr int rep,cnr tfa int op,cn met}. We mention below two relevant results concerning the Gabor matrix of a metaplectic operator and the boundedness on modulation spaces.
\begin{theorem}\label{met op properties}
	Consider $\muS \in \Mp$ and $g,\gamma\in \cS(\rd)$. For any $N\ge 0$ we have
	\[ |\la \muS \pi(z)g,\pi(w)\gamma \ra| \lesssim_{N,S} (1+|w-Sz|)^{-N}, \quad w,z \in \rdd. \]
As a consequence, for any $1\le p \le \infty$ and $s \in \bR$, the operator $\muS$ is bounded from $M^p_{v_s}(\rd)$ into itself.  
\end{theorem} 

General families of operators characterized by the sparsity of their phase-space representation were introduced in \cite{cgnr gen met,cgnr fio alg}. Given $S \in \Sp$ and $g \in \cS(\rd)$, we say that a linear operator $A : \cS(\rd) \to \cS'(\rd)$ is in the class $FIO(S)$ of \textit{generalized metaplectic operators} if there exists $H \in L^1(\rdd)$ such that 
\begin{equation} \label{def gen met}
|\la A \pi(z)g,\pi(w)g \ra | \le H(w-Sz), \quad w,z \in \rdd.
\end{equation}
The definition of $FIO(S)$ does not depend on the choice of $g \in \cS(\rd)\setminus\{0\}$. In fact, careful inspection of the proof of \cite[Prop. 3.1]{cgnr gen met} reveals that the class of admissible windows may be extended to $M^1(\rd)$, hence the estimate \eqref{def gen met} is equivalent to its \emph{polarized} version with two arbitrary windows $g,\gamma\in M^{1}(\rd)$, that is,
\begin{equation} \label{def gen met pol}
|\la A \pi(z)g,\pi(w)\gamma \ra | \le H(w-Sz), \quad w,z \in \rdd.
\end{equation}

Sparsity of the Gabor matrix of generalized metaplectic operators provides non-trivial algebraic properties for $FIO(S)$ in the spirit of Theorem \ref{weyl properties}, as detailed in the following result.  
\begin{theorem}\label{gen met properties} Let $S,S_1,S_2 \in \Sp$. 
	\begin{enumerate}
		\item An operator $T \in FIO(S)$ is bounded on $M^p(\rd)$ for any $1\le p \le \infty$.
		\item If $T_1 \in FIO(S_1)$ and $T_2 \in FIO(S_2)$, then $T_1T_2 \in FIO(S_1S_2)$.
		\item If $T \in FIO(S)$ is invertible on $L^2(\rd)$ then $T^{-1} \in FIO(S^{-1})$. 
		\item $T \in FIO(S)$ if and only if there exist $a_1,a_2 \in M^{\infty,1}(\rdd)$ such that 
		\[ T = a_1^{\w} \muS = \muS  a_2^{\w}. \] In particular, $a_2 = a_1 \circ S$. 
	\end{enumerate}
\end{theorem}
 
In view of Theorem \ref{met op properties} we observe that the Gabor matrix $\la \mu(S)\pi(z) g, \pi(w)\gamma\ra$ of a metaplectic operator $\mu(S)\in \Mp$ is well defined in the case where 
\begin{equation}\label{indcond} 
g\in M^p(\rd),\,\gamma\in M^q(\rd),\quad \frac1p+\frac1q \ge 1.
\end{equation}
To be precise, \[  \|\mu(S)\pi(z)g\|_{M^p}\leq \|\mu(S)\|_{op}\|\pi(z)g\|_{M^p}=\|\mu(S)\|_{op}\|g\|_{M^p},\quad z\in\rd,\]
hence by Proposition \ref{mod sp prop} $(iv)$
\begin{align*}|\la \mu(S)\pi(z)g,\pi(w)\gamma\ra|&\leq \|\mu(S)\|_{op}\|g\|_{M^p}\|\pi(w)\gamma\|_{M^{p'}}\\&\leq \|\mu(S)\|_{op}\|g\|_{M^p}\|\gamma\|_{M^{p'}}\\
&\leq \|\mu(S)\|_{op}\|g\|_{M^p}\|\gamma\|_{M^{q}},
\end{align*}
since from \eqref{indcond} we infer $q\leq p'$ and the inclusion $M^{q}(\rd)\subset M^{p'}(\rd)$ (Proposition \ref{mod sp prop} $(iii)$) yields the last inequality.

The same arguments apply to the Gabor matrix of $T \in FIO(S)$ as a consequence of Theorem \ref{gen met properties}. 
 
\subsection{Technical lemmas}
A key technical tool for the main results is the following set of estimates. 

\begin{lemma}\label{lem impr est}
	Let $s > 1$, $a,b,\sigma \ge 1$ and $v \in \bR$. Then
	\begin{equation}\label{lem uj 1}
	\int_{\bR} (a+|\sigma^{-1}u+v|)^{-s}(b+|u|)^{-s}du \lesssim_s (a+|v|)^{-s}b^{-s+1} + a^{-s+1}(b+|v|)^{-s+1},
	\end{equation}
	\begin{equation}\label{lem uj 2}
	\int_{\bR} (a+|u-v|)^{-s}(b+\sigma^{-1}|u|)^{-s} du \lesssim_s (a+\sigma^{-1}|v|)^{-s+1}b^{-s+1} + a^{-s+1}(b+\sigma^{-1}|v|)^{-s}.
	\end{equation}
\end{lemma}
\begin{proof}
	We prove \eqref{lem uj 1} under the assumption $|v|\ge 1$, otherwise the estimate is trivial since $\int_{\bR}(b+|u|)^{-s}du \lesssim b^{-s+1}$. If $|\sigma^{-1}u+v| \ge |v|/2$ then  $(a+|\sigma^{-1}u+v|)^{-s} \lesssim (a+|v|)^{-s}$, hence
	\[ \int_{\bR} (a+|\sigma^{-1}u+v|)^{-s}(b+|u|)^{-s}du \lesssim (a+|v|)^{-s}b^{-s+1}. \]
	If $|\sigma^{-1}u+v| \le |v|/2$ then $|u| \ge \sigma |v|/2$, hence
	\begin{align*}
	\int_{\bR} (a+|\sigma^{-1}u+v|)^{-s}(b+|u|)^{-s}du & \lesssim a^{-s+1}\sigma (b+\sigma|v|)^{-s} \\
	& \le  a^{-s+1} (b+|v|)^{-s+1}.
	\end{align*}
	The proof of \eqref{lem uj 2} in the non-trivial case $\sigma^{-1}|v|\ge 1$ follows by similar arguments, by considering separately the cases $|u-v|\ge |v|/2$ and $|u-v|<|v|/2$.
\end{proof}

\begin{remark}\label{rem conv ineq} For $s>d$ and $v \in \rd$, the convolution inequality \eqref{lem uj 2} with $\sigma =1$ and $a=b$ can be improved in $\rd$ as follows:
	\begin{equation}\label{conv ineq}
		\int_{\rd} (a+|u-v|)^{-s}(a+|u|)^{-s} du \lesssim_s a^{-s+d}(a+|v|)^{-s}.
	\end{equation}
Notice that for $a=1$ we have $v_{-s}*v_{-s} \lesssim v_{-s}$, cf.\ \cite[Lem. 11.1.1(c)]{gro book}.
\end{remark} 

\begin{lemma}\label{lem multi fine}
Let $\sigma_1, \ldots, \sigma_d \ge 1$ and define the matrices 
\[ \Sigma = \diag(\sigma_1,\ldots,\sigma_d), \quad D' =\Sigma^{-1} \oplus I, \quad D''= I\oplus \Sigma^{-1}. \] For any $s>2d$ 
\[ \int_{\rdd} (1+|v-D''u|)^{-s}(1+|D'u|)^{-s}du \lesssim_s (1+|D'v|)^{-s+2d}, \quad v \in \rdd. \] 
\end{lemma}
\begin{proof} The integral under our attention is
\[ \int_{\rdd} \lc 1+\sum_{j=1}^d |v_j-u_j| + \sum_{j=d+1}^{2d} |v_j-\sigma_{j-d}^{-1}u_j|\rc^{-s}   \lc 1+\sum_{j=1}^d |\sigma_j^{-1}u_j| + \sum_{j=d+1}^{2d} |u_j|\rc^{-s} du. \]
We look at the latter as an iterated integral and we repeatedly apply Lemma \ref{lem impr est}; precisely we estimate each of the integrals with respect to $u_1,\ldots,u_d$ as in \eqref{lem uj 2} and the each one with respect to $u_{d+1},\ldots,u_{2d}$ as in \eqref{lem uj 1}. Careful inspection of the involved quantities reveals that the result after $2d$ steps is dominated by a sum of products of the form $A^{-s+2d}B^{-s+2d}$ with $A,B \ge 1$ such that \[ A+B = 2+ \sum_{j=1}^d \sigma_j^{-1}|v_j| + \sum_{j=d+1}^{2d}|v_j| > 1+|D'v|. \] The claim follows after noticing that $A^{-s+2d}B^{-s+2d} \le (A+B)^{-s+2d}$ since $s>2d$.
\end{proof}

\section{Proof of the main results}\label{sec maint}
We start this section with the proof of Theorem \ref{maint schw}, namely a pointwise inequality for the Gabor matrix with Gabor atoms in the Schwartz class. 

\begin{proof}[Proof of Theorem \ref{maint schw}] We use the Moyal formula \eqref{moyal}, the covariance property of Wigner distribution \eqref{wig cov}  and the symplectic covariance of the Weyl calculus \eqref{symp cov}. Hence
	\begin{align*}
	\left| \la \muS \pi(z)g, \pi(w)\gamma \ra \right|^2 & = \int_{\rdd} W(\muS \pi(z) g)(u) W(\pi(w)\gamma)(u)du \\ 
	& = \int_{\rdd}  W (\pi(z)g) (S^{-1}u) W\gamma(u-w) du \\
	& = \int_{\rdd} Wg (S^{-1}u-z) W\gamma(u-w)du \\
	& = \int_{\rdd} Wg (S^{-1}u + S^{-1}w - z) W\gamma(u)du. 
\end{align*} 
Direct application of Proposition \ref{mod sp prop} $(ii)$  yields, for any $s\ge 0$, 
\[\left| \la \muS \pi(z)g, \pi(w)\gamma \ra \right|^2 \lesssim \int_{\rdd} v_{-s}(S^{-1}u + S^{-1}w - z) v_{-s}(u) du.\]
Recall that $S=U^\top DV$, hence  $S^{-1}=V^{\top}D^{-1}U$ and therefore
\[\left| \la \muS \pi(z)g, \pi(w)\gamma \ra \right|^2 \lesssim \int_{\rdd} v_{-s}(D^{-1}u + V(S^{-1}w-z)) v_{-s}(u) du.\]
Set $v\coloneqq V(S^{-1}w-z)$. The change of variable $u=D'' u'$ leads to
\[ \left| \la \muS \pi(z)g, \pi(w)\gamma \ra \right|^2 \lesssim (\det\Sigma)^{-1} \int_{\rdd} (1+|D'u+v|)^{-s}(1+|D''u|)^{-s}du.  \]
We fix $s>2d$ and apply Lemma \ref{lem multi fine} with $D'$ and $D''$ interchanged. The claim then follows after setting $N=(s-2d)/2$, since $s>2d$ is arbitrarily chosen and $D''v = D'U(w-Sz)$. \end{proof} 

We now prove Theorem \ref{maint mp}, where Gabor atoms in suitable modulation spaces are considered.

\begin{proof}[Proof of Theorem \ref{maint mp}] Fix $\phi,\psi \in \cS(\rd)\smo$ with $\lV \phi \rV_{L^2} = \lV \psi \rV_{L^2} = 1$; the reconstruction formula \eqref{rec form cont} applied to $g\in M^p(\rd)$, $\gamma\in  M^q(\rd)$ (resp. $g, \gamma \in \minftys$) yields
	\[ g = \int_{\rdd} F(u)\pi(u)\phi du, \quad F=V_\phi g \in L^p(\rdd)\,\,  (\mbox{resp.} \,\,F=V_\phi g \in L^\infty_{v_s}(\rdd) )\]
	\[ \gamma = \int_{\rdd} G(v)\pi(v)\psi dv, \quad G=V_\psi\gamma \in L^q(\rdd)\,\, (\mbox{resp.} \,\, G=V_\psi \gamma \in L^\infty_{v_s}(\rdd) ). \]
	Then we have
	\begin{align*}
	\left| \la \muS \pi(z)g, \pi(w)\gamma \ra \right| & \le \int_{\bR^{4d}} |F(u)| |G(v)| |\la \muS \pi(z+u)\phi,\pi(w+v)\psi \ra| du dv \\
	& = \int_{\bR^{4d}}  |F(u-z)| |G(v-w)| |\la \muS \pi(u)\phi,\pi(v)\psi \ra| du dv.
	\end{align*}
	Direct application of Theorem \ref{maint schw}  with $N>\max\{2d,s\}$ (the reason of this choice will be clear in a moment) yields
	\[ \left| \la \muS \pi(z)g, \pi(w)\gamma \ra \right| 
	\lesssim_N (\det\Sigma)^{-1/2} \int_{\bR^{4d}}  |F(u-z)| |G(v-w)| v_{-N}(D'U v - D''V u) du dv. \]
	Set $\wt{F} = F\circ (D''V)^{-1}$ and $\wt{G} = G\circ (D'U)^{-1}$. Then
	\begin{align*}
	\left| \la \muS \pi(z)g, \pi(w)\gamma \ra \right| 
	& \lesssim_N (\det\Sigma)^{3/2} \int_{\bR^{4d}}  |\wt{F}(u-D''Vz)| |\wt{G}(v-D'U w)| v_{-N}(v-u) du dv \\
	& = (\det\Sigma)^{3/2} \int_{\bR^{4d}}  |\wt{F}(u)| |\wt{G}(v+D''Vz-D'U w)| v_{-N}(v-u) du dv \\
	& = (\det\Sigma)^{3/2} (|\wt{F}| * v_{-N} * |\wt{G}|^{\vee})(D'U w - D''V z) \\
	& = H(D'U(w-Sz)), 
	\end{align*}
	where we defined 
	\[ H(u) = (\det\Sigma)^{3/2} (v_{-N} * |\wt{F}| * |\wt{G}|^{\vee}) (u), \quad u \in \rdd. \] 
For $g \in M^p(\rd)$ and $\gamma \in M^q(\rd)$ we apply Young's inequality to prove that $H \in L^r(\rdd)$ for $1/p+1/q=1+1/r$, cf.\ \eqref{indcond}. In particular, since $N>2d$, 
	\begin{align*} \| H\|_{L^r} & \le (\det\Sigma)^{3/2} \|v_{-N}\|_{L^1} \||\wt{F}| * |\wt{G}|^{\vee} \|_{L^r} \\ & \lesssim (\det\Sigma)^{3/2-1/p-1/q} \lV F \rV_{L^p} \lV G \rV_{L^q} \\ & \lesssim (\det\Sigma)^{1/2-1/r} \lV g \rV_{M^p} \lV \gamma \rV_{M^q}. \end{align*}

For $g,\gamma \in \minftys$, $s>2d$, we note that
\[ |\wt{F}(u)| \le \lV g \rV_{\minftys}(1+|(D'')^{-1}u|)^{-s}, \quad |\wt{G}(u)| \le \lV \gamma \rV_{\minftys}(1+|(D')^{-1}u|)^{-s}.  \]
Therefore, since $N>s$ and again by Young's inequality,
	\begin{align*} \| H\|_{L^\infty_{v_{s-2d}}} & \le (\det\Sigma)^{3/2} \|v_{-N}\|_{L^1_{v_{s-2d}}} \||\wt{F}| * |\wt{G}|^{\vee} \|_{L^{\infty}_{v_{s-2d}}} \\ & \lesssim (\det\Sigma)^{3/2} \| g \|_{\minftys} \| \gamma \|_{\minftys} \lV v_{-s}((D')^{-1}\cdot) \ast v_{-s}((D'')^{-1}\cdot)\rV_{L^\infty_{v_{s-2d}}}\\ & \lesssim (\det\Sigma)^{-1/2} \lV g \rV_{\minftys} \lV \gamma \rV_{\minftys},    \end{align*}
where in the last step we used Lemma \ref{lem multi fine} with the substitutions $u\mapsto (D')^{-1}(D'')^{-1}u$ and $v \mapsto (D')^{-1}v$.
\end{proof}

\begin{remark} \label{rem maint diag}
	Notice that after setting $\wt{H} = H\circ D'U$ the estimate \eqref{eq maint mp} reads
	\[ 	\left| \la \muS \pi(z)g, \pi(w)\gamma \ra \right| \le \wt{H}(w-Sz), \] 
	while \eqref{eq norm maint mp} becomes
	\[ \| \wt{H} \|_{L^r} \lesssim (\det\Sigma)^{1/2} \lV g \rV_{M^p} \lV \gamma \rV_{M^q}. \] 
	It is then clear that there is a trade-off between phase-space concentration of $\muS$ along the graph of $S$ and the spreading of wave packets.
\end{remark}

We conclude with a result in the same spirit for generalized metaplectic operators. 

\begin{theorem}\label{maint gen met}
	Let $1\le p,q,r \le \infty$ satisfy $1/p + 1/q = 1+ 1/r$. Consider $S\in \Sp$ with an  Euler decomposition  $(U,V,\Sigma)$, $a \in M^{\infty,1}(\rdd)$ so that $T\coloneqq a^\w \muS \in FIO(S)$, cf.\ Theorem \ref{gen met properties}. For any $g \in M^p(\rd)$, $\gamma \in M^q(\rd)$ there exists $H \in L^r(\rdd)$ such that, for any $z,w \in \rdd$,
\begin{equation}
\left| \la T \pi(z)g, \pi(w)\gamma \ra \right| \le H(D'U(w-Sz)),
\end{equation} 
with 
\[ \lV H \rV_{L^r} \le  (\det\Sigma)^{1/2-1/r} \lV a \rV_{M^{\infty,1}} \lV g \rV_{M^p} \lV \gamma \rV_{M^q}. \] 
\end{theorem}

\begin{proof}
We assume $\| g\|_{L^2}=\|\gamma\|_{L^2}=1$ without loss of generality. Denoting by $K_{\muS}(w,z)=\la \muS \pi(z)g, \pi(w)\gamma \ra$ the Gabor matrix of $\muS$ and similarly for $K_{a^\w}(w,z)=\la a^\w \pi(z)\gamma,\pi(w)\gamma\ra$, in view of \eqref{gabor ker}, by Theorems \ref{weyl properties} and \ref{gen met properties} we have 
\begin{align*}
|\la T \pi(z)g, \pi(w)\gamma \ra| & = |\la a^\w \muS \pi(z)g, \pi(w)\gamma \ra| \\
& \le \int_{\rdd} |K_{a^\w}(w,u)| |K_{\muS}(u,z)| du \\
& = \int_{\rdd} |\la a^\w \pi(u)\gamma, \pi(w)\gamma \ra| |\la \muS \pi(z)g, \pi(u)\gamma \ra|  du\\
& \le \int_{\rdd} H_a(w-u) H_S(D'Uu-D''Vz)du, \end{align*} where $H_S$ is the controlling function in Theorem \ref{maint mp} $(i)$ and $H_a$ is the one appearing in Theorem \ref{weyl properties} with $g=\gamma$; in particular $\lV H_a \rV_{L^1} \asymp \lV a \rV_{M^{\infty,1}}$. The substitution $y = D'U(w-u)$ yields
\[ |\la T \pi(z)g, \pi(w)\gamma \ra| \le (\det\Sigma) \left[ (H_a\circ (D'U)^{-1})*H_S\right] (D'U(w-Sz)). \]

The claim follows by Young inequality and Theorem \ref{maint mp} $(i)$ after setting $H=(\det\Sigma) (H_a\circ (D'U)^{-1})*H_S$:
\begin{align*} \lV H \rV_{L^r} & \le (\det\Sigma) \lV H_a\circ (D'U)^{-1}\rV_{L^1} \lV H_S\rV_{L^r} \\
& = \lV H_a\rV_{L^1} \lV H_S\rV_{L^r} \\ & \lesssim (\det\Sigma)^{1/2-1/r} \lV a \rV_{M^{\infty,1}} \lV g \rV_{M^p}\lV\gamma\rV_{M^q}.
\end{align*}
\end{proof}

We conclude with the proof of Theorem \ref{maint prop sing}; namely we study how the modulation space regularity on a cone in the phase space behaves under the action of a metaplectic operator. 

\begin{proof}[Proof of Theorem \ref{maint prop sing}]
	Fix $g,\gamma \in \cS(\rd)\smo$ with $\lV g \rV_{L^2} = \| \gamma \|_{L^2} = 1$, and $\Gamma$ and $\Gamma'$ as in the statement. From \eqref{gabor ker} with $A=\muS$ and Theorem \ref{maint schw}, for any $N>0$ we have
	\begin{align*}
	|V_{\gamma} (\muS f)(w)| & \le \int_{\rdd} |K_{\muS}(w,z)| |V_{g} f(z)|dz \\
	& \lesssim_N (\det\Sigma)^{-1/2} \int_{\rdd} v_{-N}(D'U(w-Sz)) |V_{g} f(z)|dz \\ 
	& \lesssim_N (\det\Sigma)^{-1/2} \int_{\rdd} H(w-Sz) |V_{g} f(z)|dz, 
	\end{align*}
where we set $H = v_{-N} \circ D' U$.  After naming $G = H\circ S = v_{-N} \circ D''V$ we apply H\"older's inequality and get
	\begin{align*}
I & \coloneqq \lV \mu(S)f \rV_{M^1_{(\gamma)}(S(\Gamma'))} \\ & = \int_{S(\Gamma')}|V_{\gamma} (\muS f)(w)| dw\\  & = \int_{\Gamma'} |V_{\gamma} (\muS f)(Sw)| dw \\
& \lesssim (\det\Sigma)^{-1/2} \int_{\Gamma'} \int_{\rdd} G(w-z) |V_{g} f(z)|dz dw.
\end{align*}
We then have $I \lesssim I_1+I_2$, where 
\[ I_1 \coloneqq  (\det\Sigma)^{-1/2}\int_{\Gamma'}\int_{\Gamma} G(w-z) |V_{g} f(z)| dzdw,  \]
\[ I_2 \coloneqq (\det\Sigma)^{-1/2}\int_{\Gamma'}\int_{\Gamma^c} G(w-z) |V_{g} f(z)|dzdw.  \]
Young's inequality yields
\begin{equation*} I_1  \le \lV G \rV_{L^1} \lV V_{g} f \cdot 1_{\Gamma} \rV_{L^1} \lesssim (\det\Sigma)^{1/2} \lV f \rV_{M^1_{(g)}(\Gamma)}. 
\end{equation*}
After setting $F(z) = |V_{g}f(z)|v_{-r}(z)$, the remaining integral is
\[ I_2 = (\det\Sigma)^{-1/2}\int_{\Gamma'}\int_{\Gamma^c} G(w-z)v_{r}(z) F(z)dz. \] The key point is now that
\[  1+|w-z| \asymp \max\{1+|w|, 1+|z|\}, \quad w \in \Gamma',\, z\in \Gamma^c,  \] hence
\begin{align*} I_2 & \lesssim (\det\Sigma)^{-1/2}\int_{\Gamma'}\int_{\Gamma^c} G(w-z)v_{r}(w-z) F(z)dz \\ 
& \le (\det\Sigma)^{-1/2} \lV (G\cdot v_r)*F \rV_{L^1} \\
& \lesssim (\det\Sigma)^{-1/2} \lV G\cdot v_r \rV_{L^1} \lV f \rV_{M^1_{v_{-r}}}. \end{align*}
Therefore, the remaining integral to estimate is 
\[ \lV G\cdot v_r \rV_{L^1} = \int_{\rdd} (1+|D''z|)^{-N} (1+|z|)^r dz.  \]
Recall that $D''=I\oplus \Sigma^{-1}$, cf.\ \eqref{def D'D''}, and consider the elementary estimates
\[ v_{-N}(D''z) \le v_{-N/2d}(z_1)\cdots v_{-N/2d}(z_d) v_{-N/2d}(\sigma_1^{-1}z_{d+1})\cdots v_{-N/2d}(\sigma_d^{-1}z_{2d}), \]
\[ v_r(z) \le v_{r}(z_1)\cdots v_r(z_{2d}). \]
As a result, the integral is dominated by $A^dB_1 \cdots B_d$, where 
\[ A\coloneqq \int_{\bR} (1+|x|)^{-N/2d + r} dx, \] \[ B_j \coloneqq \int_{\bR} (1+\sigma_j^{-1}|x|)^{-N/2d}(1+|x|)^r dx, \quad j=1,\ldots,d. \] 
If $N$ is large enough then $A < \infty$ and $B_j \lesssim \sigma_j^{1+r}$, therefore 
\[ I_2 \lesssim (\det\Sigma)^{1/2+r} \lV f \rV_{M^1_{v_{-r}}}, \] 
and the claim follows.
\end{proof}

\begin{remark} \begin{enumerate}
\item Condition \eqref{M1 cone def} can be generalized to introduce the notion of $M^p$-regularity, $1 \le p \le \infty$, on the cone $\Gamma$ with respect to $g \in \cS(\rd)\smo$. The latter is satisfied for $f\in \cS'(\rd)$ if
\begin{equation}\label{Mp cone def} \lV f \rV_{M^p_{(g)}(\Gamma)} \coloneqq \lV V_g f \cdot 1_{\Gamma} \rV_{L^p} < \infty. \end{equation} Weighted versions of such conditions can be defined similarly. The proof of Theorem \ref{maint prop sing} can be easily modified in order to prove the estimate 
\begin{equation}\label{Mp prop sing est} \lV \muS f \rV_{M^p_{(\gamma)}(S (\Gamma'))} \lesssim (\det\Sigma)^{1/2} \lc \lV f \rV_{M^p_{(g)}(\Gamma)} + (\det\Sigma)^{r}\lV f \rV_{M^p_{v_{-r}}} \rc, \end{equation}
which however is not sharp unless $p=1$ or $p=\infty$. We postpone further investigations on the issue to a subsequent paper. \\
\item The notion of $M^p$-regularity does not depend on the window $g$ used to compute $V_g f$ in \eqref{Mp cone def} provided that a slightly smaller cone is allowed when changing window. This is indeed a consequence of \eqref{Mp prop sing est} in the case where $S=I$. The properties of $M^p_{(g)}(\Gamma)$ as a function space will be object of future studies. 
\end{enumerate}  
\end{remark}

\begin{corollary}\label{cor34}
Consider $1\le p \le \infty$. There exists $C>0$ such that, for any $f \in M^p(\rd)$, $S\in\Sp$,	\[ \| \mu(S)f \|_{M^p} \le C (\det\Sigma)^{|1/2-1/p|}\| f \|_{M^p}. \] 
\end{corollary}
\begin{proof}
	By choosing $\Gamma=\Gamma'=\rdd\smo$ and $r=0$ in Theorem \ref{maint prop sing} we see that the desired estimate holds for $p=1$. Since $\mu(S)$
	is unitary on $L^2(\rd)$, the operator $\mu(S^{-1})$, and therefore $\mu(S)$, satisfies the same estimate for $p=\infty$. Interpolating with the trivial $L^2$-estimate, we obtain the desired result (modulation spaces interpolate like the corresponding $L^p$ spaces \cite{F1}).
\end{proof}

\section{Applications to the free particle propagator}\label{sec free prop}

Let us consider the free particle propagator $U(t)=e^{i(t/2\pi)\triangle}$ and the corresponding classical flow \eqref{schr prop}; a straightforward computation shows that the largest $d$ singular values  of $\cA_t$ coincide: \[ \sigma_j = \sigma(t) = (1+2t^2 + 2(t^2+t^4)^{1/2})^{1/2} = \sqrt{1+t^2}+|t|, \quad j=1,\ldots,d. \] 
Note in particular that $\sigma(t)$ is comparable to $ 1+|t|$, $t \in \bR$. An example of Euler decomposition  $(U_t,V_t,\Sigma_t)$ of $\cA_t$ for $t\ge0$ is given by 
\[ U_t = (1+\sigma(t)^2)^{-1/2} \begin{bmatrix}
\sigma(t)I & I \\ -I & \sigma(t)I
\end{bmatrix}, \quad V_t = (1+\sigma(t)^2)^{-1/2} \begin{bmatrix} I & \sigma(t)I \\ -\sigma(t)I & I \end{bmatrix}. \] Theorem \ref{maint schw} thus yields
\[ 	\left| \la e^{i(t/2\pi)\triangle} \pi(z)g, \pi(w)\gamma \ra \right| \le C (1+|t|)^{-d/2} (1+  |D'_tU_t(w-\cA_t z)|)^{-N}, \quad z,w \in \rdd.  \]
The spreading phenomenon manifests itself as a dilation by 
\[ D'_tU_t = (1+\sigma(t)^2)^{-1/2}\begin{bmatrix} I & \sigma(t)^{-1}I \\ -I & \sigma(t)I \end{bmatrix}. \]
We attempt to shed some light on the apparently unintelligible structure of such matrix by means of a toy example in dimension $d=1$. Let $z=0$ for simplicity and assume that the atom $g$ is concentrated on the box $Q=\{ (x,\xi)\in \bR^2 : |x|<1,\, |\xi|<1 \}$ in the time-frequency plane. In view of \eqref{maint estimate} we are lead to consider
\[ (D'_tU_t)^{-1}(Q) = \{ (x,\xi) : |x+\sigma^{-1}(t)\xi|<\sqrt{1+\sigma(t)^2}, \, |x-\sigma(t)\xi| < \sqrt{1+\sigma(t)^2} \}. \]
Therefore, the effect of $D'_t U_t$ on $Q$ ultimately amounts to a horizontal stretch by a factor of approximately $\sigma(t)$. This is consistent with the expected phase-space evolution of a wave packet, as shown in \eqref{intro At free}; see also \cite[Fig. 1-8]{cnr spars gabor} for illuminating graphic representations. We stress that the estimate \eqref{intro diag est schr} is completely blind to such spreading effect.

\section*{Acknowledgments.} The authors are members of the Gruppo Nazionale per l'Analisi Matematica, la Probabilit\`a e le loro Applicazioni (GNAMPA) of the Istituto Nazionale di Alta Matematica (INdAM). The present research was partially supported by MIUR grant “Dipartimenti di Eccellenza” 2018–2022, CUP: E11G18000350001, DISMA, Politecnico di Torino.

\end{document}